\documentclass{amsart}
\usepackage{amssymb,amsmath,amsthm}

\newtheorem{corollary}{Corollary}

\newtheorem{theorem}{Theorem}
\newtheorem{lemma}{Lemma}

\numberwithin{equation}{section}

\begin{document}

\title[General convolution identities]
{General convolution identities for Bernoulli and Euler polynomials}

\author{Karl Dilcher}
\address{Department of Mathematics and Statistics\\
 Dalhousie University\\
         Halifax, Nova Scotia, B3H 4R2, Canada}
\email{dilcher@mathstat.dal.ca}

\author{Christophe Vignat}
\address{Department of Mathematics\\
Tulane University\\
New Orleans, LA 70118}
\email{cvignat@tulane.edu}
\thanks{The first author was supported in part by the Natural Sciences and 
Engineering Research Council of Canada}


\setcounter{equation}{0}

\begin{abstract}
Using general identities for difference operators, as well as a technique of 
symbolic computation and tools from probability theory, we derive very general 
$k$th order ($k\geq 2$) convolution identities for Bernoulli and Euler 
polynomials. This is achieved by use of an elementary result on
uniformly distributed random variables.
These identities depend on $k$ positive real parameters, and as special cases
we obtain numerous known and new identities for these 
polynomials. In particular we show that the well-known identities of
Miki and Matiyasevich for Bernoulli numbers are special cases of the same
general formula.
\end{abstract}

\maketitle

\section{Introduction}

The Bernoulli and Euler numbers and polynomials have been studied extensively 
over the last two centuries, both for their numerous important applications 
in number theory, combinatorics, numerical analysis and other areas of pure 
and applied mathematics, and for their rich structures as interesting objects
in their own right. The Bernoulli numbers $B_n$, $n=0,1,2,\ldots$, can be 
defined by the exponential generating function
\begin{equation}\label{1.1}
\frac{z}{e^z-1} = \sum_{n=0}^\infty B_n\frac{z^n}{n!}\qquad (|z|< 2\pi).
\end{equation}
They are rational numbers, the first few being 1, $-\frac{1}{2}$, $\frac{1}{6}$,
0, $-\frac{1}{30}$, 0, $\frac{1}{42},\ldots$, with $B_{2k+1}=0$ for $k\geq 1$.
For the most important properties see, for instance, \cite[Ch.~23]{AS} or its 
successor \cite[Ch.~24]{DLMF}. Other good references are \cite{GKP}, 
\cite{Jo}, or \cite{No}. For a general bibliography, see \cite{DSS}.

Numerous linear and nonlinear recurrence relations for these numbers 
are known, and such relations also exist for the Bernoulli {\it polynomials}
and for Euler numbers and polynomials which will be defined later.
This paper deals with {\it nonlinear\/} recurrence relations, the prototype
of which is Euler's well-known identity
\begin{equation}\label{1.2}
\sum_{j=0}^n \binom{n}{j}B_jB_{n-j} = -nB_{n-1} - (n-1)B_n\qquad (n\geq 1).
\end{equation}
This can also be seen as a convolution identity. Two different types of 
convolution identities were discovered more recently, namely
\begin{equation}\label{1.3}
\sum_{j=2}^{n-2}\frac{B_jB_{n-j}}{j(n-j)} 
-\sum_{j=2}^{n-2}\binom{n}{j}\frac{B_jB_{n-j}}{j(n-j)}
= 2H_n\frac{B_n}{n} \qquad (n\geq 4)
\end{equation}
by Miki \cite{Mi}, where $H_n=1+\frac{1}{2}+\dots+\frac{1}{n}$ is the $n$th 
harmonic number, and
\begin{equation}\label{1.4}
(n+2)\sum_{j=2}^{n-2}B_jB_{n-j}
-2\sum_{j=2}^{n-2}\binom{n+2}{j}B_jB_{n-j} = n(n+1)B_n \qquad (n\geq 4)
\end{equation}
by Matiyasevich \cite{Ma}; see also \cite{Ag} and the references therein.
These two identities, which are remarkable in that they combine two different 
types of convolutions, were later extended to Bernoulli polynomials by Gessel 
\cite{Ge2} and by Pan and Sun \cite{PS}, respectively. Gessel \cite{Ge2} also
extended \eqref{1.3} to third-order convolutions, i.e., sums of products
of three Bernoulli numbers. Later Agoh \cite{Ag} found different and simpler 
proofs of the polynomial analogues of \eqref{1.3} and \eqref{1.4} and proved 
numerous other similar identities involving Bernoulli, Euler, and Genocchi 
numbers and polynomials. Subsequently Agoh and the first author \cite{AD}
extended the polynomial analogue of \eqref{1.4} to convolution identities of
arbitrary order, and did the same for Euler polynomials. Meanwhile, following
different lines of investigation, Dunne and Schubert \cite{DS} derived
an identity that has both \eqref{1.3} and \eqref{1.4} as special cases, and
Chu \cite{Ch} obtained a large number of convolution identities, some of them 
extending \eqref{1.3} and \eqref{1.4}.

It is the purpose of this paper to contribute to the recent work summarized
above and to further extend the identities \eqref{1.3} and \eqref{1.4} of
Miki and Matiyasevich. In Section~2 we state a general result concerning 
second-order convolutions, and derive some consequences. In Section~3 we
introduce a symbolic notation with a related calculus, and use it to state and
prove a very general identity for Bernoulli polynomials. This is then used
in Section~4, along with some methods from probability theory, to prove a
general higher-order convolution identity which gives the main result of 
Section~2 as a special case. In Section~5 we apply most of the methods from
Sections~3 and~4 to Euler polynomials and again obtain general higher-order
convolution identities. Finally, in Section~6, we state and prove several
further consequences of each of our main theorems. We conclude this paper with
some further remarks in Section~7.

\section{Identities for Bernoulli polynomials}

The {\it Bernoulli polynomials\/} can be defined by 
\begin{equation}\label{2.1}
B_n(x):=\sum_{j=0}^n\binom{n}{j}B_jx^{n-j},
\end{equation}
or equivalently by the generating function
\begin{equation}\label{2.2}
\frac{ze^{xz}}{e^z-1}=\sum_{n=0}^\infty B_n(x)\frac{z^n}{n!}\qquad (|z|< 2\pi).
\end{equation}
For the first few Bernoulli polynomials, see Table~1 in Section~5.
They have the special values
\begin{equation}\label{2.3}
B_n(0) = B_n,\qquad B_n(\tfrac{1}{2})=(2^{1-n}-1)B_n,\qquad B_n(1)=(-1)^nB_n,
\end{equation}
($n=0, 1, 2,\ldots$),
where the first identity is immediate from comparing \eqref{2.2} with 
\eqref{1.1}, and the other two follow from easy manipulations of the generating
function \eqref{2.2}. We also require the Pochhammer symbol (or rising 
factorial) $(z)_k$, defined for $z\in\mathbb C$ and integers $k\geq 0$ by
\begin{equation}\label{2.4}
(z)_k = \frac{\Gamma(z+k)}{\Gamma(z)} = z(z+1)\dots (z+k-1),
\end{equation}
where the right-hand product is valid for $k\geq 1$.

We are now ready to state our first main result, which will be proved later.

\begin{theorem}\label{thm:Thm1}
For integers $n\geq 1$ and real numbers $a, b> 0$ we have
\begin{align}
\sum_{l=0}^{n}\binom{n}{l}\frac{(a)_{l}(b)_{n-l}}{(a+b)_{n}}B_{l}(x)B_{n-l}(x) 
&= \sum_{l=0}^{n}\binom{n}{l}\frac{a(b)_{l}+b(a)_{l}}{(a+b)_{l+1}}B_lB_{n-l}(x)
\label{2.5}\\
&\qquad+ \frac{ab}{(a+b+1)(a+b)}nB_{n-1}(x).\nonumber 
\end{align}
\end{theorem}

The remainder of this section will be devoted to deriving a number of 
consequences of this general identity. First, it is clear by \eqref{2.3} that 
we get an analogous identity for Bernoulli {\it numbers} by simply deleting
the variable $x$. The most immediate special case is obtained by setting 
$a=b=1$. With $(1)_n=n!$ and $(2)_n=(n+1)!$, some straightforward manipulations
involving the binomial coefficients in \eqref{2.5} lead to the following
identity, which was earlier obtained in \cite{AD}.

\begin{corollary}
For all $n\geq 1$ we have 
\begin{equation}\label{2.6}
(n+2)\sum_{l=0}^n B_l(x)B_{n-l}(x)
=2\sum_{l=0}^n\binom{n+2}{l+2}B_lB_{n-l}(x)+\binom{n+2}{3}B_{n-1}(x).
\end{equation}
\end{corollary}

When $x=0$, this identity becomes trivial for odd $n$ since one of 
$B_l, B_{n-l}$ will be zero except in the cases $l=1$ and $l=n-1$. 
For even $n\geq 4$, however, we have the following identity.

\begin{corollary}
For all even $n\geq 2$ we have 
\begin{equation}\label{2.7}
(n+2)\sum_{l=0}^n B_lB_{n-l}=2\sum_{l=0}^n\binom{n+2}{l+2}B_lB_{n-l}.
\end{equation}
\end{corollary}

This identity, although different in appearance, is equivalent to \eqref{1.4}.
For our next corollary we need the {\it shifted harmonic numbers} which for
real $a>0$ and integers $n\geq 1$ are defined by 
\begin{equation}\label{2.8}
H_{a,n} := \sum_{j=0}^{n-1}\frac{1}{j+a}.
\end{equation}
Obviously, $H_{1,n}=H_n$. As we shall see, the following result can be 
considered as an infinite class of generalizations of Miki's identity 
\eqref{1.3}.

\begin{corollary}
For real $a>0$ and integers $n\geq 1$ we have
\begin{align}
\sum_{l=0}^{n-1}\binom{n}{l}\frac{(a)_{l}(n-l-1)}{(a)_{n}}B_{l}(x)B_{n-l}(x)
&= \sum_{l=1}^{n}\binom{n}{l}\frac{a(l-1)!+(a)_{l}}{(a)_{l+1}}B_lB_{n-l}(x)
\label{2.9}\\
&\qquad+ \frac{n}{a+1}B_{n-1}(x)+H_{a,n}B_n(x).\nonumber
\end{align}
\end{corollary}

\begin{proof}
The idea of proof is to divide both sides of \eqref{2.5} by $b$ and then take
the limit as $b\rightarrow 0$. On the left-hand side we have for 
$0\leq l\leq n-1$,
\[
\frac{1}{b}\cdot\frac{(a)_l(b)_{n-l}}{(a+b)_n}
=\frac{(a)_l(b+1)\dots(b+n-l-1)}{(a+b)_n}
\rightarrow\frac{(a)_l(n-l-1)!}{(a)_n}
\]
as $b\rightarrow 0$, and on the right-hand side, for $1\leq l\leq n$,
\[
\frac{1}{b}\cdot\frac{a(b)_l+b(a)_l}{(a+b)_{l+1}}
=\frac{a(b+1)\dots(b+l-1)+(a)_l}{(a+b)_{l+1}}
\rightarrow\frac{a(l-1)!+(a)_l}{(a)_{l+1}}
\]
as $b\rightarrow 0$. To take care of the terms that were left out in the 
limits above, we note that the term for $l=0$ in the right-hand sum and the
term for $l=n$ in the left-hand sum of \eqref{2.5} combine to give
\[
\binom{n}{0}\frac{a+b}{a+b}B_n(x)-\frac{(a)_n}{(a+b)_n}B_n(x)
=\frac{(a+b)_n-(a)_n}{(a+b)_n}B_n(x),
\]
and we have 
\[
\lim_{b\rightarrow 0}\frac{(a+b)_n-(a)_n}{b(a+b)_n}
=\frac{1}{(a)_n}\left.\frac{d}{dx}(x)_n\right|_{x=a} = H_{a,n},
\]
where the second equation follows directly from applying the product rule
repeatedly to the right-hand side of \eqref{2.4}. Putting everything 
together, we get \eqref{2.9}.
\end{proof}

As illustrations of Corollary~3 we state the cases $a=1$ and $a=2$ separately.

\begin{corollary}
For integers $n\geq 1$ we have
\begin{align}
\frac{n}{2}\sum_{l=1}^{n-1}\frac{B_l(x)}{l}\frac{B_{n-l}(x)}{n-l}
&= \sum_{l=1}^{n}\binom{n}{l}\frac{B_l}{l}B_{n-l}(x)
+\frac{n}{2}B_{n-1}(x)+H_{n-1}B_n(x),\label{2.10} \\
(n+2)\sum_{l=0}^{n-1}(l+1)B_{l}(x)&\frac{B_{n-l}(x)}{n-l}
=\sum_{l=1}^{n}\binom{n+2}{l+2}(l^2+l+1)\frac{B_l}{l}B_{n-l}(x)\label{2.11}\\
&\qquad+(n+1)(n+2)\left(\frac{n}{3}B_{n-1}(x)+H_{2,n}B_n(x)\right).\nonumber
\end{align}
\end{corollary}

\begin{proof}
After some easy manipulations, \eqref{2.9} with $a=1$ gives
\begin{equation}\label{2.12}
\sum_{l=0}^{n-1}B_l(x)\frac{B_{n-l}(x)}{n-l}
=\sum_{l=1}^{n}\binom{n}{l}\frac{B_l}{l}B_{n-l}(x) 
+\frac{n}{2}B_{n-1}(x)+H_nB_n(x),
\end{equation}
and with $a=2$, \eqref{2.9} gives \eqref{2.11}. We now exploit the symmetry on
the left-hand side of \eqref{2.12} and rewrite the sum as
\[
\frac{B_n(x)}{n}+\frac{1}{2}\sum_{l=1}^{n-1}
\left(\frac{1}{n-l}+\frac{1}{l}\right)B_l(x)B_{n-l}(x) = \frac{B_n(x)}{n}
+\frac{n}{2}\sum_{l=1}^{n-1}\frac{B_l(x)}{l}\frac{B_{n-l}(x)}{n-l}.
\]
Finally we subtract $\frac{1}{n}B_n(x)$ from both sides of \eqref{2.12} and 
note that $H_nB_n(x)$ then becomes $H_{n-1}B_n(x)$.
\end{proof}

Using a technique that involves generating functions for Stirling numbers
and N\"orlund polynomials, Gessel \cite{Ge2} obtained \eqref{2.10} as a 
polynomial analogue of Miki's identity \eqref{1.3}. When $x=0$, then we also
have symmetry in the sum on the right-hand side of \eqref{2.10}; we can 
therefore use again the identity $1/(n-l)+1/l=n/l(n-l)$, upon which we easily
recover Miki's identity.

Some further consequences of Theorem~1 and Corollary~3 will be derived in 
the final section of this paper.

\section{Symbolic notation and general identities}

{\bf 1.} The use of symbolic notation in dealing with Bernoulli numbers and
polynomials goes back to J.~Blissard in the 1860s. Subsequently it was used
by many other authors, among them \'E.~Lucas in the 1870s and 1880s. Later it
was put on a firm theoretical foundation as part of ``the classical umbral
calculus"; see, e.g., \cite{Ge1} or \cite{RT}.

Here we propose and use a system of symbolic notation that is in some respects
similar to the classical umbral calculus, but is different in that it is 
related to probability theory. Also, this system of notation is more specific
to Bernoulli numbers and polynomials and (later in this paper) Euler numbers
and polynomials.

The basis for our symbolic notation for Bernoulli numbers and polynomials are
two symbols, $\mathcal{B}$ and $\mathcal{U}$, which are complementary to each
other or, as we shall see, annihilate each other. First, we define the
{\it Bernoulli symbol\/} $\mathcal{B}$ by
\begin{equation}\label{3.1}
\mathcal{B}^{n}=B_{n}\qquad (n=0, 1,\ldots)
\end{equation}
so that, for instance, \eqref{2.1} can be rewritten as
\begin{equation}\label{3.1a}
B_n(x) = (x+\mathcal{B})^n.
\end{equation}
Furthermore, with \eqref{1.1} we have
\begin{equation}\label{3.2}
\exp\left(\mathcal{B}z\right)=\sum_{n=0}^\infty\mathcal{B}^n\frac{z^n}{n!}
=\frac{z}{e^z-1}.
\end{equation}
We also require several independent Bernoulli symbols 
$\mathcal{B}_{1},\dots,\mathcal{B}_{k}$. Independence means that if we have 
any two Bernoulli symbols, say $\mathcal{B}_1$ and $\mathcal{B}_2$, then
\begin{equation}\label{3.2a}
\mathcal{B}_1^{k}\mathcal{B}_2^{\ell}=B_kB_\ell.
\end{equation}
Second, the {\it uniform symbol\/} $\mathcal{U}$ is defined by
\begin{equation}\label{3.3}
f(x+\mathcal{U})=\int_0^1f(x+u)du.
\end{equation}
Here and elsewhere we assume that $f$ is an analytic function for which the
objects in question exist. From \eqref{3.3} we immediately obtain, in 
analogy to \eqref{3.1},
\begin{equation}\label{3.4}
\mathcal{U}^{n}=\frac{1}{n+1}\qquad(n=0, 1,\ldots),
\end{equation}
and using this, we get
\begin{equation}\label{3.5}
\exp\left(\mathcal{U}z\right)=\sum_{n=0}^\infty\mathcal{U}^n\frac{z^n}{n!}
=\frac{e^z-1}{z}.
\end{equation}
From \eqref{3.2} and \eqref{3.5} we now deduce
\[
\exp\left(z\left(\mathcal{B}+\mathcal{U}\right)\right)
=\sum_{n=0}^\infty\left(\mathcal{B}+\mathcal{U}\right)^n\frac{z^n}{n!} = 1,
\]
which means that $\mathcal{B}$ and $\mathcal{U}$ annihilate each other, 
i.e., $(\mathcal{B}+\mathcal{U})^n=0$ for all $n\neq 0$, in the sense that
\begin{equation}\label{3.6}
f(x+\mathcal{B}+\mathcal{U})=f(x),
\end{equation}
or in other words, we have the equivalence
\begin{equation}\label{3.7}
f(x)=g(x+\mathcal{U})\quad\Leftrightarrow\quad g(x)=f(x+\mathcal{B}).
\end{equation}
Finally, we note that \eqref{3.3} immediately gives, for any $u\in\mathbb R$,
\begin{equation}\label{3.8}
uf'(x+u\mathcal{U}) = f(x+u)-f(x),
\end{equation}
a difference equation that will be used repeatedly.

\medskip
{\bf 2.} It is well known that the Bernoulli polynomials are closely related
to the calculus of finite differences; see, e.g., the classic books \cite{Jo}
or \cite{No}. It is therefore not surprising that methods from difference
calculus turn out to be useful in the proofs of our main results.
Let $\Delta_u$ be the {\it forward difference operator\/} defined by
\begin{equation}\label{3.9}
\Delta_uf(x)=f(x+u)-f(x).
\end{equation}
With two (in general) distinct differences $u_1, u_2$ we compute
\begin{align*}
\Delta_{u_1}\Delta_{u_2}f(x)&=\left(f(x+u_2+u_1)-f(x+u_1)\right)
-\left(f(x+u_2)-f(x)\right) \\
&=\left(f(x+u_1+u_2)-f(x)\right)-\left(f(x+u_1)-f(x)\right)
-\left(f(x+u_2)-f(x)\right)\\
&=\Delta_{u_1+u_2}f(x)-\Delta_{u_1}f(x)-\Delta_{u_2}f(x),
\end{align*}
which gives the operator identity
\begin{equation}\label{3.10}
\Delta_{u_1+u_2}=\Delta_{u_1}\Delta_{u_2}+\Delta_{u_1}+\Delta_{u_2}.
\end{equation}
Similarly, one obtains
\begin{align}
\Delta_{u_1+u_2+u_3} &= \Delta_{u_1}\Delta_{u_2}\Delta_{u_3}
+\Delta_{u_1}\Delta_{u_2}+\Delta_{u_1}\Delta_{u_3}+\Delta_{u_2}\Delta_{u_3}\label{3.11} \\
&\quad +\Delta_{u_1}+\Delta_{u_2}+\Delta_{u_3}.\nonumber
\end{align}
To generalize these identities, we use the following notation: For a fixed 
integer $k\geq 1$ and for any subset $J\subseteq\{1,\ldots,k\}$, we denote
\begin{equation}\label{3.12}
\Delta_J:=\prod_{j\in J}\Delta_{u_j},
\end{equation}
and we let $|J|$ be the cardinality of $J$. We can now state and prove the
following simple but important lemma.

\begin{lemma}
For any $k\geq 1$ and for real numbers $u_1,\ldots,u_k$ we have
\begin{equation}\label{3.13}
\Delta_{u_1+\dots+u_k}=\sum_{j=1}^k\sum_{|J|=j}\Delta_J.
\end{equation}
\end{lemma}

The case $k=1$ is trivial, and we immediately see that $k=2$ and $k=3$ give 
the identities \eqref{3.10} and \eqref{3.11}, respectively.

\begin{proof}[Proof of Lemma~1]
This result can be proved by induction on $k$ in a straightforward way.
Alternatively, and more formally, we can use the shift operator
\[
f(x+u) = e^{u\partial}f(x),
\]
with the differential operator $\partial=\frac{d}{dx}$. Then we have
$\Delta_u=e^{u\partial}-1$, and
\[
\Delta_{u_1+\dots+u_k}=e^{(u_{1}+\dots+u_{k})\partial}-1
=\sum_{j=1}^k\sum_{|J|=j}\prod_{\ell\in J}\left(e^{u_{\ell}\partial}-1\right),
\]
and the result follows.
\end{proof}

{\bf 3.} We now apply results from the first two parts of this section to 
obtain a general identity for Bernoulli symbols, and thus for Bernoulli numbers
and polynomials. In what follows, we assume that for a fixed integer $k\geq 1$,
$u_1,\ldots,u_k$ are real parameters. To simplify notations, we write, for a
subset $J\subseteq\{1,\ldots,k\}$,
\begin{equation}\label{3.14}
u_{J} := \prod_{j\in J}u_{j},\qquad 
\left(u\mathcal{B}\right)_{J} := \sum_{j\in J}u_{j}\mathcal{B}_{j},\qquad
\overline{J}=\{1,\dots,k\}\setminus J.
\end{equation}
The following is, in fact, a restatement of an intermediate result in \cite{AD}.

\begin{lemma}
Let $u_1+\dots+u_k=1$. Then we have
\begin{equation}\label{3.15}
\frac{1}{n!}\left(x+u_1\mathcal{B}_1+\dots+u_k\mathcal{B}_k\right)^n
=\sum_{j=1}^k\sum_{|J|=j}\frac{u_J}{(n+1-j)!}
\left(x+\mathcal{B}_0+(u\mathcal{B})_{\overline{J}}\right)^{n-j+1},
\end{equation}
where $\mathcal{B}_0,\dots,\mathcal{B}_k$ are independent Bernoulli symbols.
\end{lemma}

\begin{proof}
We apply the operator identity \eqref{3.13} to the function
\[
f(x):=\frac{1}{(n+1)!}
\left(x+\mathcal{B}_0+u_1\mathcal{B}_1+\dots+u_k\mathcal{B}_k\right)^{n+1}.
\]
Then the left-hand side of \eqref{3.13} gives, with \eqref{3.8},
\begin{align}
\Delta_{u_1+\dots+u_k}f(x) &= \Delta_1f(x)=f(x+1)-f(x)\label{3.16}\\
&= f'(x+\mathcal{U})
=\frac{1}{n!}\left(x+u_1\mathcal{B}_1+\dots+u_k\mathcal{B}_k\right)^n,\nonumber
\end{align}
where in the last step we used \eqref{3.6}, i.e., $\mathcal{B}_0$ is 
annihilated by $\mathcal{U}$. Similarly, we have for any $i=1,\ldots, k$, again
using \eqref{3.8},
\begin{align}
\Delta_{u_i}f(x) &= f(x+u_i)-f(x) = u_if'(x+u_i\mathcal{U}) \label{3.17}\\
&= u_i\frac{1}{n!}\left(x+\mathcal{B}_0
+(u\mathcal{B})_{\{1,\ldots,k\}\setminus\{i\}}\right)^n,\nonumber
\end{align}
having used the fact that the uniform symbol $\mathcal{U}$ annihilated the
Bernoulli symbol $\mathcal{B}_i$; note that the coefficients $u_i$ have to
match for the annihilation (i.e., identity \eqref{3.6}) to apply. Using the
definition \eqref{3.12} and successively applying \eqref{3.17}, we get
\[
\Delta_Jf(x)=\frac{u_J}{(n-j+1)!}
\left(x+\mathcal{B}_{0}+(u\mathcal{B})_{\overline{J}}\right)^{n+1-j}.
\]
Finally, applying \eqref{3.13} to this and to \eqref{3.16}, we immediately
get \eqref{3.15}.
\end{proof}

While the case $k=1$ is trivial, for $k=2$ and $k=3$ we get the following
identities.

\medskip
\noindent
{\bf Examples.} For $u_1+u_2=1$, we have
\begin{align*}
\left(x+u_1\mathcal{B}_1+u_2\mathcal{B}_2\right)^n 
&= u_1\left(x+\mathcal{B}_0+u_2\mathcal{B}_2\right)^n
+u_2\left(x+\mathcal{B}_0+u_1\mathcal{B}_1\right)^n\\
&\quad +u_1u_2n\left(x+\mathcal{B}_0\right)^{n-1},
\end{align*}
and for $u_1+u_2+u_2=1$,
\begin{align*}
&\left(x+u_1\mathcal{B}_1+u_2\mathcal{B}_2+u_3\mathcal{B}_3\right)^n 
= \left[u_1\left(x+\mathcal{B}_0+u_2\mathcal{B}_2+u_3\mathcal{B}_3\right)^n+o.t.\right]\\
&\qquad+\left[nu_1u_2\left(x+\mathcal{B}_0+u_3\mathcal{B}_3\right)^{n-1}+o.t.\right]
+n\left(n-1\right)u_1u_2u_3\left(x+\mathcal{B}_0\right)^{n-2},
\end{align*}
where ``o.t." in each of the first and second rows stands for the ``other terms"
obtained by cyclically permuting the subscripts $\{1,2,3\}$.

\medskip
\noindent
{\bf Remarks.} (1) The left-hand side of \eqref{3.15}, and in fact also the
terms on the right-hand side, could be written as Bernoulli polynomials of
higher order, as defined in identity (30) in \cite[p.~39]{EMOT}. We will not
pursue this further.

(2) It is clear from the proof of Lemma~2 that, more generally, 
for any analytic function $f$ and $u_1+\dots+u_k=1$ we have
\[
f(x+u_1\mathcal{B}_1+\dots+u_k\mathcal{B}_k)
=\sum_{j=1}^n\sum_{|J|=j}u_J
f^{(j-1)}\left(x+\mathcal{B}_0+(u\mathcal{B})_{\overline{J}}\right),
\]
and in particular, for $k=2$ and $u_1+u_2=1$,
\begin{align*}
f(x+u_1\mathcal{B}_1+u_2\mathcal{B}_2) 
&=u_2f(x+\mathcal{B}_0+u_1\mathcal{B}_1)+u_1f(x+\mathcal{B}_0+u_2\mathcal{B}_2)\\
&\quad +u_1u_2f'(x+\mathcal{B}_0).
\end{align*}

The main results of this paper are based on Lemma~2 and an analogue for Euler
polynomials, and will be obtained by considering the expectation when
$u_1,\ldots,u_k$ are taken to be certain random variables. 

\section{Generalization and proof of Theorem~1}

{\bf 1.} In this section we prove a higher-order analogue of Theorem~1, of 
which the latter is an immediate consequence. The proof uses some probabilistic
methods which will be summarized in a brief subsection.

\begin{theorem}
For integers $k\geq 2$ and $n\geq 0$ and for positive real parameters 
$a_1,\ldots,a_k$ we have
\begin{align}
&\sum_{l_1+\dots+l_k=n}\binom{n}{l_1,\ldots,l_k}
\frac{(a_1)_{l_1}\dots(a_k)_{l_k}}{(a_1+\dots+a_k)_n}B_{l_1}(x)\dots B_{l_k}(x)
=\sum_{j=1}^k\sum_{|J|=j}\frac{a_J n!}{(n+1-j)!}\label{4.0} \\
&\qquad\times\sum_{\substack{l_0+l_1+\dots+l_{k-j}\\=n+1-j}}\binom{n+1-j}{l_0,l_1,\ldots,l_{k-j}}
\frac{(a_{i_{j+1}})_{l_1}\dots(a_{i_k})_{l_{k-j}}}{(a_1+\dots+a_k)_{n+1-l_0}}
B_{l_0}(x)B_{l_1}\dots B_{l_{k-j}}.\nonumber
\end{align}
\end{theorem}

When $k=2$, we immediately get Theorem~1. For $k=3$ and $a_1=a_2=a_3=1$ we get,
after some easy transformations and renaming the summation indices,
\begin{gather}
(n+3)\sum_{i+j+l=n}B_i(x)B_j(x)B_l(x)
= 3\sum_{i+j+l=n}\binom{n+3}{i}B_i(x)B_jB_l\label{4.0a}\\
\qquad+3\sum_{i+j=n-1}\binom{n+3}{i}B_i(x)B_j
+\binom{n+3}{5}B_{n-2}(x),\nonumber
\end{gather}
valid for $n\geq 3$; this is Corollary~1 in \cite{AD}. Other special cases 
with $k=3$ will be considered
later, in Section~6. For arbitrary $k\geq 2$, with $a_1=\dots=a_k=1$, we
recover Theorem~1 in \cite{AD}, which for $x=0$ gives a $k$th order analogue
of Matiyasevich's identity \eqref{1.4}, namely 
\[
\sum_{l_1+\dots+l_k=n}B_{l_1}\dots B_{l_k}
=\frac{1}{n+k}\sum_{j=1}^k\binom{k}{j}
\sum_{\substack{l_0+l_1+\dots+l_{k-j}\\=n+1-j}}\binom{n+k}{l_0}
B_{l_0}B_{l_1}\dots B_{l_{k-j}}.
\]

{\bf 2.} We now summarize some facts from probability theory that will be
used in the proofs that follow. For the basics we refer the reader to any 
introductory text in probability theory, e.g., \cite{Du} or \cite{Wa}.
For the interplay between probability theory and umbral calculus, see 
\cite{SW}.

We assume that $X$ is a continuous random variable with probability density
function $f_X(x)$, i.e.,
\begin{equation}\label{4.1}
{\rm Pr}(X\leq x) = \int_{-\infty}^x f_X(y)dy.
\end{equation}
Given a measurable function $g:{\mathbb R}\rightarrow {\mathbb R}$ such that
the image random variable $g(X)$ is absolutely integrable, its expectation
can be expressed as
\begin{equation}\label{4.2}
{\mathbb E}g(X) = \int_{-\infty}^\infty g(y)f_X(y)dy.
\end{equation}
The main tool in this section is the use of random variables with a gamma
distribution of ``scale parameter" 1. We write such a random variable as
$X \sim\Gamma_a$, with ``shape parameter" $a>0$, defined by the density function
\begin{equation}\label{4.3}
f_X(x;a) = \begin{cases}
\tfrac{1}{\Gamma(a)}x^{a-1}e^{-x} &\hbox{for}\; x\geq 0,\\
0 &\hbox{otherwise}.
\end{cases}
\end{equation}
Then from the definition of the gamma function,
\[
\Gamma(s) = \int_0^\infty x^{s-1}e^{-x}dx,
\]
and with \eqref{4.2} and \eqref{4.3} we immediately get, for an integer 
$n\geq 1$,
\begin{equation}\label{4.4}
{\mathbb E}(\Gamma_a^n) = \int_0^\infty y^n\tfrac{1}{\Gamma(a)}y^{a-1}e^{-y}dy
=\frac{\Gamma(a+n)}{\Gamma(a)} = (a)_n.
\end{equation}
An essential property of the gamma distribution is additivity, i.e., if 
$\Gamma_{a_1},\ldots,\Gamma_{a_k}$ are independent gamma distributed random
variables, then
\begin{equation}\label{4.5}
\Gamma_{a_1}+\dots+\Gamma_{a_k} \sim \Gamma_{a_1+\dots+a_k},
\end{equation}
where the symbol $\sim$ indicates that the random variables on both sides have
the same distribution. The relation \eqref{4.5} follows from the fact that the
density probability function for the sum of two independent random variables
is the convolution of the individual ones; see, e.g., \cite[p.~107]{Wa}. 

The next important tool is the choice of random coefficients $u_1,\ldots, u_k$
such that $(u_1,\ldots, u_k)$ follows a Dirichlet distribution with parameters
$(a_1,\ldots, a_k)$. This is equivalent to choosing $k$ independent gamma
random variables $\Gamma_{a_i}$, each having shape parameter $a_i$, and to 
define
\begin{equation}\label{4.6}
u_i=\frac{\Gamma_{a_i}}{\Gamma_{a_1}+\dots+\Gamma_{a_k}},\qquad 1\leq i\leq k;
\end{equation}
note that $u_1+\dots+u_k=1$.
For Dirichlet distributions in general, see, e.g., \cite[p.~231]{JK}. We now 
need an important property of gamma random variables, namely that
$\Gamma_a+\Gamma_b$ and $\Gamma_a/(\Gamma_a+\Gamma_b)$ are independent when
$\Gamma_a$ and $\Gamma_b$ are. In fact, this characterizes gamma random 
variables; see \cite{Lu}. This is easily extended to the
statement that
\begin{equation}\label{4.7}
\Gamma_{a_1}+\dots+\Gamma_{a_k}\quad\hbox{and}\quad
\frac{\Gamma_{a_i}}{\Gamma_{a_1}+\dots+\Gamma_{a_k}},\qquad 1\leq i\leq k,
\end{equation}
are independent. The importance of this lies in the fact that 
${\mathbb E}(XY)={\mathbb E}(X){\mathbb E}(Y)$ for independent random 
variables $X$ and $Y$. 

Combining all of the above, we first note that for any positive integers
$l_1,\ldots, l_k$ we have by \eqref{4.6},
\begin{align}
{\mathbb E}[(\Gamma_{a_1}+\dots+\Gamma_{a_k})^{l_1+\dots+l_k}
(u_1^{l_1}\dots u_k^{l_k})]
&={\mathbb E}(\Gamma_{a_1}^{l_1}\dots\Gamma_{a_k}^{l_k})\label{4.8}\\
&={\mathbb E}(\Gamma_{a_1}^{l_1})\dots
{\mathbb E}(\Gamma_{a_k}^{l_k}).\nonumber
\end{align}
On the other hand, using the independence of the terms in \eqref{4.7}, we see
that the left-hand side of \eqref{4.8} is equal to 
\begin{equation}\label{4.9}
{\mathbb E}(\Gamma_{a_1+\dots+a_k}^{l_1+\dots+l_k})
{\mathbb E}(u_1^{l_1}\dots u_k^{l_k}),
\end{equation}
having also used \eqref{4.5}. Finally, applying \eqref{4.4} to the right-hand
side of \eqref{4.8} and to \eqref{4.9}, we get
\begin{equation}\label{4.10}
{\mathbb E}(u_1^{l_1}\dots u_k^{l_k})
=\frac{(a_1)_{l_1}\dots (a_k)_{l_k}}{(a_1+\dots+a_k)_{l_1+\dots+l_k}}.
\end{equation}
This identity will be used repeatedly in what follows.

\medskip
{\bf 3.} We are now ready to prove Theorem~2; as we shall see, much of the 
work was already done in obtaining the identities \eqref{3.15} and \eqref{4.10}.

\begin{proof}[Proof of Theorem~2]
We choose $u_1,\ldots, u_k$ as in \eqref{4.6}. Since $u_1+\dots+u_k=1$, we can
rewrite the $n$th power term on the left-hand side of \eqref{3.15} as follows,
and then apply a multinomial expansion, using \eqref{3.1a}:
\begin{align}
&\left(u_1(x+\mathcal{B}_1)+\dots+u_k(x+\mathcal{B}_k)\right)^n\label{4.11}\\
&\qquad\qquad=\sum_{l_1+\dots+l_k=n}\binom{n}{l_1,\ldots,l_k}
u_1^{l_1}\dots u_k^{l_k}B_{l_1}(x)\dots B_{l_k}(x).\nonumber
\end{align}
Similarly, we use multinomial expansions for the powers on the right of 
\eqref{3.15}, this time combining the terms $x+\mathcal{B}_0$ for the sake of
applying \eqref{3.1a}:
\begin{align}
&\left(x+\mathcal{B}_0+(u\mathcal{B})_{\overline{J}}\right)^{n-j+1}\label{4.12}\\
&\qquad=\sum_{\substack{l_0+l_1+\dots+l_{k-j}\\=n+1-j}}\binom{n+1-j}{l_0,l_1,\ldots,l_{k-j}}
B_{l_0}(x)\left(u_{i_{j+1}}\mathcal{B}_{i_{j+1}}\right)^{l_1}\dots
\left(u_{i_k}\mathcal{B}_{i_k}\right)^{l_{k-j}}\nonumber\\
&\qquad=\sum_{\substack{l_0+l_1+\dots+l_{k-j}\\=n+1-j}}\binom{n+1-j}{l_0,l_1,\ldots,l_{k-j}}
u_{i_{j+1}}^{l_1}\dots u_{i_k}^{l_{k-j}}B_{l_0}(x)B_{l_1}\dots B_{l_{k-j}},\nonumber
\end{align}
where we have also used \eqref{3.2a}. All that remains to be done now is to 
compute the expectation on both sides of \eqref{3.15}, which mainly involves
applying \eqref{4.10} to the right-hand sides of \eqref{4.11} and \eqref{4.12}.
In particular, keeping the first notation in \eqref{3.14} in mind, we have
\begin{align*}
{\mathbb E}(u_Ju_{i_{j+1}}^{l_1}\dots u_{i_k}^{l_{k-j}})
&=\frac{(a_{i_1})_1\dots(a_{i_j})_1(a_{i_{j+1}})_{l_1}\dots(a_{i_k})_{l_{k-j}}}
{(a_1+\dots+a_k)_{n+1-l_0}} \\
&=a_J\frac{(a_{i_{j+1}})_{l_1}\dots(a_{i_k})_{l_{k-j}}}{(a_1+\dots+a_k)_{n+1-l_0}},
\end{align*}
where we have used the fact that $(a)_1=a$ and, in the denominator, that 
$1+\dots +l_1+\dots+l_{k-j}=j+(n-j+1)-l_0=n+1-l_0$. This completes the proof. 
\end{proof}

\section{Euler numbers and polynomials}

The Euler numbers and polynomials are often considered in parallel with their
Bernoulli analogues. Indeed, they are similar in various respects, including
their importance in the classical calculus of finite differences (see, e.g.,
\cite{Jo} or \cite{No}). In this section we follow the outlines of the previous
sections to derive analogous results for Euler polynomials and, to a lesser
extent, Euler and Genocchi numbers.

\medskip
{\bf 1.} The {\it Euler numbers\/} $E_n$, $n=0, 1, 2,\ldots$, 
can be defined by 
\begin{equation}\label{5.1}
\frac{2}{e^z+e^{-z}} = \sum_{n=0}^\infty E_n\frac{z^n}{n!}\qquad 
(|z|< \tfrac{\pi}{2}).
\end{equation}
The Euler numbers are all integers with $E_n=0$ when $n$ is odd; the first
few values are listed in Table~1. The {\it Euler polynomials\/} can be 
defined by
\begin{equation}\label{5.2}
E_n(x):=\sum_{j=0}^n\binom{n}{j}\frac{E_j}{2^j}(x-\tfrac{1}{2})^{n-j},
\end{equation}
or equivalently by the generating function
\begin{equation}\label{5.3}
\frac{2e^{xz}}{e^z+1}=\sum_{n=0}^\infty E_n(x)\frac{z^n}{n!}\qquad (|z|< \pi).
\end{equation}
A key consequence of \eqref{5.3} is the functional equation
\begin{equation}\label{5.4}
E_n(x)+E_n(x+1)=2x^n,\qquad n=0,1,2,\ldots,
\end{equation}
which gives rise to numerous applications.
One important difference to the Bernoulli case is the fact that $E_n(0)$ is 
{\it not\/} the $n$th Euler number. The {\it Genocchi numbers\/} $G_n$, are 
often used instead; they are closely related to the Bernoulli numbers via
\begin{equation}\label{5.5}
G_n := 2(1-2^n)B_n\qquad (n=0, 1, 2,\ldots).
\end{equation}
These numbers are all integers; the first few values are also listed 
in Table~1.

\bigskip
\begin{center}
{\renewcommand{\arraystretch}{1.2}
\begin{tabular}{|r||r|r|r|l|l|}
\hline
$n$ & $B_n$ & $E_n$ & $G_n$ & $B_n(x)$ & $E_n(x)$\\
\hline
0 & 1 & 1 & 0 & 1 & 1 \\
1 & $-1/2$ & 0 & 1 & $x-\tfrac{1}{2}$ & $x-\tfrac{1}{2}$ \\
2 & $1/6$ & $-1$ & $-1$ & $x^2-x+\tfrac{1}{6}$ & $x^2-x$ \\
3 & 0 & 0 & 0 & $x^3-\tfrac{3}{2}x^2+\tfrac{1}{2}x$ & $x^3-\tfrac{3}{2}x^2+\tfrac{1}{4}$ \\
4 & $-1/30$ & 5 & 1 & $x^4-2x^3+x^2-\tfrac{1}{30}$ &
$x^4-2x^3+x$ \\
5 & 0 & 0 & 0 & $x^5-\tfrac{5}{2}x^4+\tfrac{5}{3}x^3-\tfrac{1}{6}x$ &
$x^5-\tfrac{5}{2}x^4+\tfrac{5}{2}x^2-\tfrac{1}{2}$ \\
6 & $1/42$ & $-61$ & $-3$ & $x^6-3x^5+\tfrac{5}{2}x^4-\tfrac{1}{2}x^2+\tfrac{1}{42}$ &
$x^6-3x^5+5x^3-3x$ \\
\hline
\end{tabular}}

\medskip
{\bf Table~1}: $B_n, E_n, G_n, B_n(x)$ and $E_n(x)$ for $0\leq k\leq 6$.
\end{center}
\bigskip

By elementary manipulations of the relevant generating functions, we get
\begin{equation}\label{5.6}
E_n(0)=\frac{1}{n+1}G_{n+1},\qquad E_n(\tfrac{1}{2})=2^{-n}E_n \qquad
(n=0, 1, 2,\ldots).
\end{equation}
The Euler polynomial analogue of Theorem~1 can now be stated as follows.

\begin{theorem}
For integers $n\geq 1$ and real numbers $a, b> 0$ we have
\begin{align}
\sum_{l=0}^{n}\binom{n}{l}&\frac{(a)_{l}(b)_{n-l}}{(a+b)_{n}}E_{l}(x)E_{n-l}(x)
= \frac{4}{n+1}B_{n+1}(x) \label{5.6a}\\ 
&-\frac{2}{n+1}\sum_{l=0}^{n+1}\binom{n+1}{l}
\frac{(a)_{l}+(b)_{l}}{(a+b)_l}E_l(0)B_{n+1-l}(x)\nonumber
\end{align}
\end{theorem}
As in the case of Theorem~1, this result follows from a higher-order 
convolution identity that will be proved later. As a special case of
\eqref{5.6a}, for $a=b=1$, we get the following Euler polynomial analogue
of Matiyasevich's identity:
\[
(n+2)\sum_{l=0}^{n}\binom{n}{l}E_{l}(x)E_{n-l}(x)
=4(n+2)B_{n+1}(x)-4\sum_{l=0}^{n+1}\binom{n+2}{l}B_l(x)E_{n+1-l}(0).
\]
This identity was earlier obtained as Corollary~2 in \cite{AD}.

\medskip
{\bf 2.} As we develop a formalism parallel to that involving the Bernoulli
symbol, we note that the analogue of $B_n$ is $E_n(0)$. Thus, we define the
{\it Euler symbol} $\mathcal E$ by
\begin{equation}\label{5.7}
{\mathcal E}^n = E_n(0),\qquad n=0,1,2,\ldots,
\end{equation}
and elementary manipulation of the generating function \eqref{5.3} gives
\begin{equation}\label{5.8}
E_n(x) = (x+{\mathcal E})^n,\qquad n=0,1,2,\ldots.
\end{equation}
The analogue to the uniform symbol $\mathcal U$ defined in Section~3 is the 
uniform discrete symbol $\mathcal V$ with generating function
\begin{equation}\label{5.9}
e^{z{\mathcal V}} = \tfrac{1}{2}+\tfrac{1}{2}e^z,
\end{equation}
or equivalently defined by 
\[
f(x+\mathcal{V}) = \frac{f(x)+f(x+1)}{2}
\]
for an analytic function $f$; this is a discrete analogue of \eqref{3.3}.
With a change of variable we have for any real $u$,
\begin{equation}\label{5.10}
f(x+u{\mathcal V}) = \frac{1}{2}f(x)+\frac{1}{2}f(x+u),
\end{equation}
which is analogous to \eqref{3.8}, and which will be just as useful. Next,
by multiplying \eqref{5.3}, setting $x=0$, with \eqref{5.9}, we see that in
analogy with \eqref{3.6} and \eqref{3.7} we have ${\mathcal E}+{\mathcal V}=0$
in the sense that 
\begin{equation}\label{5.11}
f(x+{\mathcal E}+{\mathcal V}) = f(x),
\end{equation}
or in other words,
\begin{equation}\label{5.12}
f(x)=g(x+\mathcal{V})\quad\Leftrightarrow\quad g(x)=f(x+\mathcal{E}).
\end{equation}

{\bf 3.} The functional equations \eqref{5.4} and \eqref{5.10} give rise to 
the definition of the {\it discrete forward difference operator} $\delta_u$ 
defined by 
\begin{equation}\label{5.13}
\delta_uf(x) = \frac{f(x)+f(x+u)}{2}.
\end{equation}
Thus, in particular, we have $\delta_1E_n(x)=x^n$ and by \eqref{5.10},
\begin{equation}\label{5.14}
\delta_uf(x) = f(x+u{\mathcal V}).
\end{equation}
In analogy to \eqref{3.10} we now
compute
\begin{align*}
2\delta_{u_1}\delta_{u_2}f(x)
&=\frac{f(x)+f(x+u_2)}{2}+\frac{f(x+u_1)+f(x+u_2+u_1)}{2} \\
&=\frac{f(x+)+f(x+u_1+u_2)}{2}+\frac{f(x)+f(x+u_1)}{2}\\
&\qquad+\frac{f(x)+f(x+u_2)}{2}-f(x), 
\end{align*}
which gives the operator identity
\begin{equation}\label{5.15}
\delta_{u_1+u_2}=2\delta_{u_1}\delta_{u_2}-\delta_{u_1}-\delta_{u_2}+1.
\end{equation}
Similarly, one obtains
\begin{align}
\delta_{u_1+u_2+u_3} &= 4\delta_{u_1}\delta_{u_2}\delta_{u_3}
-2\delta_{u_1}\delta_{u_2}-2\delta_{u_1}\delta_{u_3}-2\delta_{u_2}\delta_{u_3}\label{5.16} \\
&\qquad +\delta_{u_1}+\delta_{u_2}+\delta_{u_3}.\nonumber
\end{align}
In general, using the notation $\delta_J$, with the same meaning as in
\eqref{3.12}, where again $J\subseteq\{1,\ldots,k\}$, we have the following
result.

\begin{lemma}
For even $k\geq 2$ we have
\begin{equation}\label{5.17}
\delta_{u_1+\dots+u_k} = 1-\sum_{j=1}^k\sum_{|J|=j}(-2)^{j-1}\delta_J,
\end{equation}
and for odd $k\geq 1$,
\begin{equation}\label{5.18}
\delta_{u_1+\dots+u_k} = \sum_{j=1}^k\sum_{|J|=j}(-2)^{j-1}\delta_J.
\end{equation}
\end{lemma}

These identities can be proved by straightforward induction, with \eqref{5.15}
as induction beginning. Using notation from \eqref{3.14}, we now
obtain the following result.

\begin{lemma}
Let $u_1+\dots+u_k=1$. Then for even $k\geq 2$ we have
\begin{equation}\label{5.19}
(n+1)\left(x+u_1\mathcal{E}_1+\dots+u_k\mathcal{E}_k\right)^n
=\sum_{j=1}^k(-2)^j\sum_{|J|=j}
\left(x+\mathcal{B}+(u\mathcal{E})_{\overline{J}}\right)^{n+1},
\end{equation}
and for odd $k\geq 1$,
\begin{equation}\label{5.20}
\left(x+u_1\mathcal{E}_1+\dots+u_k\mathcal{E}_k\right)^n
=\sum_{j=1}^k(-2)^{j-1}\sum_{|J|=j}
\left(x+\mathcal{E}_0+(u\mathcal{E})_{\overline{J}}\right)^n,
\end{equation}
where $\mathcal{E}_0,\dots,\mathcal{E}_k$ are independent Euler symbols.
\end{lemma}

\begin{proof}
Comparing \eqref{5.13} with \eqref{3.9}, we get the operator identity
$\Delta_u = 2\delta_u-2$,
and thus for even $k\geq 2$ we have with \eqref{5.17},
\begin{equation}\label{5.22}
\Delta_{u_1+\dots+u_k} = \sum_{j=1}^k(-2)^j\sum_{|J|=j}\delta_J.
\end{equation}
Since $u_1+\dots u_k=1$, we have by \eqref{3.16},
\begin{equation}\label{5.23}
\Delta_{u_1+\dots+u_k}f(x) = f'(x+\mathcal{U}). 
\end{equation}
Now let 
\[
f(x):=\left(x+\mathcal{B}+u_1\mathcal{E}_1+\dots+u_k\mathcal{E}_k\right)^{n+1},
\]
and apply \eqref{5.22} to this function. On the left-hand side, using 
\eqref{5.23}, the symbols $\mathcal{U}$ and $\mathcal{B}$ cancel each other,
and we get the left-hand side of \eqref{5.19}. To obtain the right-hand side,
we first note that for any $i=1,\ldots,k$ we have by \eqref{5.14},
\begin{equation}\label{5.24}
\delta_{u_i}f(x) 
=\left(x+\mathcal{B}+(u\mathcal{E})_{\{1,\ldots,k\}\setminus\{i\}}\right)^{n+1},
\end{equation}
having used the notation in \eqref{3.14} and the fact that the symbols 
$\mathcal{V}$ and $\mathcal{E}_i$ cancel each other. As in
\eqref{3.17}, the coefficients $u_i$ have to match for this cancellation to
apply. Successively applying \eqref{5.24} and using the notation \eqref{3.12},
we get
\[
\delta_Jf(x) = \left(x+\mathcal{B}+(u\mathcal{E})_{\overline{J}}\right)^{n+1}.
\]
This, combined with \eqref{5.22}, completes the proof of \eqref{5.19}.

The proof of \eqref{5.20} is very similar: Instead of \eqref{5.22} use
\eqref{5.18} and apply it to 
\[
f(x):=\left(x+\mathcal{E}_0+u_1\mathcal{E}_1+\dots+u_k\mathcal{E}_k\right)^n.
\]
While the right-hand side is evaluated as before, for the left-hand side we
use \eqref{5.14} with $u=1$.
\end{proof}

{\bf 4.} We are now ready to state and prove the main result of this section.

\begin{theorem}
Let $n\geq 0$ and $k\geq $ be integers and $a_1,\ldots,a_k$ positive 
parameters. Then for even $k\geq 2$ we have
\begin{gather}
\sum_{l_1+\dots+l_k=n}\binom{n}{l_1,\ldots,l_k}
\frac{(a_1)_{l_1}\dots(a_k)_{l_k}}{(a_1+\dots+a_k)_n}E_{l_1}(x)\dots E_{l_k}(x)
=\sum_{j=1}^k\frac{(-2)^j}{n+1}\label{5.25} \\
\quad\times\sum_{|J|=j}\sum_{\substack{l_0+l_1+\dots\\+l_{k-j}=n+1}}\binom{n+1}{l_0,\ldots,l_{k-j}}
\frac{(a_{i_{j+1}})_{l_1}\dots(a_{i_k})_{l_{k-j}}}{(a_1+\dots+a_k)_{n+1-l_0}}
B_{l_0}(x)E_{l_1}(0)\dots E_{l_{k-j}}(0),\nonumber
\end{gather}
and for odd $k\geq 1$,
\begin{gather}
\sum_{l_1+\dots+l_k=n}\binom{n}{l_1,\ldots,l_k}
\frac{(a_1)_{l_1}\dots(a_k)_{l_k}}{(a_1+\dots+a_k)_n}E_{l_1}(x)\dots E_{l_k}(x)
=\sum_{j=1}^k(-2)^{j-1}\label{5.26} \\
\quad\times\sum_{|J|=j}\sum_{\substack{l_0+l_1+\dots\\+l_{k-j}=n}}\binom{n}{l_0,\ldots,l_{k-j}}
\frac{(a_{i_{j+1}})_{l_1}\dots(a_{i_k})_{l_{k-j}}}{(a_1+\dots+a_k)_{n-l_0}}
E_{l_0}(x)E_{l_1}(0)\dots E_{l_{k-j}}(0).\nonumber
\end{gather}
\end{theorem}

For $k=2$, the identity \eqref{5.25} reduces to Theorem~3.
In the special case $a_1=\dots=a_k=1$, the identities \eqref{5.25} and 
\eqref{5.26} reduce to Theorems~2 and~3, respectively, in \cite{AD}. Other
special cases can be found in Section~6.

\begin{proof}[Proof of Theorem~4]
The proof is almost identical to that of Theorem~2: We expand the powers on
both sides of \eqref{5.19} and \eqref{5.20} using the multinomial theorem,
and then compute the expectation on both sides by way of \eqref{4.10}, having
chosen $u_1,\ldots,u_k$ as in \eqref{4.6}.

The left-hand sides of \eqref{5.25} and \eqref{5.26} are obtained just as in
the expansion \eqref{4.11}, with Euler instead of Bernoulli symbols and 
polynomials, and having used \eqref{5.8} in place of \eqref{3.1a}. The
right-hand sides are expanded as in \eqref{4.12}, with the appropriate
exponent and with ``Bernoulli" replaced by ``Euler" where appropriate.
Applying \eqref{4.10} then completes the proofs of both identities.
\end{proof}

\section{Some further identities}

In this final section we state and prove some further consequences of our 
main results from Sections~2, 4 and 5, respectively.

\subsection{Consequences of Theorem~1}

In addition to the two identities in Corollary~4,
we can obtain one more consequence of Corollary~3 by multiplying both sides of
\eqref{2.9} by $a$ and then taking the limit as $a\rightarrow\infty$. Then all
the terms in the sum on the left disappear, with the exception of the $l=n-1$
term. On the right, the fraction in the sum tends to 1 for $l\geq 2$, and to 2
for $l=1$. Putting everything together, we get the following consequence.

\begin{corollary}
For integers $n\geq 1$ we have
\begin{equation}\label{2.13}
\sum_{l=0}^n\binom{n}{l}B_lB_{n-l}(x) = n(x-1)B_{n-1}(x)-(n-1)B_n(x).
\end{equation}
\end{corollary}

For $x=0$ this is Euler's identity \eqref{1.2}, but it is also a special case
of identity (5.11.2) in \cite{Ha}.

We can obtain even more consequences from Theorem~1 by setting 
$a=b=\varepsilon$ and then taking the limit as $\varepsilon\rightarrow\infty$,
or by considering the terms in \eqref{2.5} as power series in $\varepsilon$.
We begin with the first case.

\begin{corollary}
For integers $n\geq 1$ we have
\begin{equation}\label{2.14}
\sum_{l=0}^n\binom{n}{l}\frac{1}{2^l}B_lB_{n-l}(x) 
= \frac{n}{2^n}(2x-1)B_{n-1}(2x)-\frac{n-1}{2^n}B_n(2x)-\frac{n}{4}B_{n-1}(x).
\end{equation}
\end{corollary}

\begin{proof}
With $a=b=\varepsilon$, the following limits are obvious:
\[
\lim_{\varepsilon\rightarrow\infty}\frac{(\varepsilon)_l(\varepsilon)_{n-l}}{(2\varepsilon)_n}=\frac{1}{2^n},\qquad
\lim_{\varepsilon\rightarrow\infty}\frac{2\varepsilon(\varepsilon)_l}{(2\varepsilon)_{l+1}}=\frac{1}{2^l},\qquad
\lim_{\varepsilon\rightarrow\infty}\frac{\varepsilon^2}{(2\varepsilon+1)(2\varepsilon)}=\frac{1}{4}.
\]
Hence we have
\begin{equation}\label{2.15}
\frac{1}{2^n}\sum_{l=0}^n\binom{n}{l}B_l(x)B_{n-l}(x)
=\sum_{l=0}^n\binom{n}{l}\frac{1}{2^l}B_lB_{n-l}(x)+\frac{n}{4}B_{n-1}(x).
\end{equation}
The sum on the left has a well-known evaluation (see (50.11.2) in \cite{Ha}):
\[
\sum_{l=0}^n\binom{n}{l}B_l(x)B_{n-l}(x)=n(2x-1)B_{n-1}(2x)-(n-1)B_n(2x).
\]
This, with \eqref{2.15}, immediately gives \eqref{2.14}.
\end{proof}

We note that \eqref{2.15} can also be obtained as a special case of identity
(6.1) in \cite{AD}. For the next statement, again using 
$a=b=\varepsilon$, we need the {\it second-order harmonic numbers\/}, defined
by $H_0^{(2)}:=0$ and 
\[
H_n^{(2)}:=\sum_{j=1}^n\frac{1}{j^2}\qquad(n\geq 1).
\]
We can now prove the following result.

\begin{corollary}
For integers $n\geq 1$ we have
\begin{align}
\sum_{l=1}^{n-1}\left(H_{n-1}-H_{l-1}\right)
\frac{B_l(x)}{l}&\frac{B_{n-l}(x)}{n-l}
=\sum_{l=1}^{n}\binom{n}{l}\left(H_l+\frac{1}{l}\right)\frac{B_l}{l}B_{n-l}(x)\label{2.16} \\
&+nB_{n-1}(x)+\frac{1}{2}\left(H_{n-1}^2+3H_{n-1}^{(2)}\right)B_n(x).\nonumber
\end{align}
\end{corollary}

\begin{proof}
Setting $a=b=\varepsilon$ in \eqref{2.5} and dividing both sides by 
$\varepsilon$, we have to expand the following terms. First, for 
$1\leq l\leq n-1$ we have 
\begin{align}
\frac{1}{\varepsilon}\frac{(a)_{l}(b)_{n-l}}{(a+b)_{n}}
&=\frac{(\varepsilon+1)\dots(\varepsilon+l-1)(\varepsilon+1)\dots(\varepsilon+n-l-1)}{2(2\varepsilon+1)\dots(2\varepsilon+n-1)}\label{2.17}\\
&=\frac{(l-1)!(n-l-1)!}{2(n-1)!}
\prod_{j=1}^{l-1}\left(1+\frac{\varepsilon}{j}\right)
\prod_{j=1}^{n-l-1}\left(1+\frac{\varepsilon}{j}\right)
\prod_{j=1}^{n-1}\left(1+\frac{2\varepsilon}{j}\right)^{-1} \nonumber\\
&=\frac{(l-1)!(n-l-1)!}{2(n-1)!}\left(1+\left(H_{l-1}+H_{n-l-1}-2H_{n-1}\right)\varepsilon+O(\varepsilon^2)\right).\nonumber
\end{align}
Next, for $l\geq 1$ we get
\begin{align}
\frac{1}{\varepsilon}\frac{a(b)_{l}+b(a)_{l}}{(a+b)_{l+1}}
&=\frac{(\varepsilon+1)\dots(\varepsilon+l-1)}{(2\varepsilon+1)\dots(2\varepsilon+l)}
=\frac{1}{l}\left(1+\frac{2\varepsilon}{l}\right)^{-1}
\prod_{j=1}^{l-1}\frac{1+\frac{\varepsilon}{j}}{1+\frac{2\varepsilon}{j}}\label{2.18}\\
&=\frac{1}{l}\left(1-\frac{2\varepsilon}{l}+\dots\right)
\prod_{j=1}^{l-1}\left(1-\frac{\varepsilon}{j}+\dots\right) \nonumber \\
&=\frac{1}{l}\left(1-\left(H_l+\frac{1}{l}\right)\varepsilon+O(\varepsilon^2)\right),\nonumber
\end{align}
where we have used the fact that $H_{l-1}+2/l=H_l+1/l$. Next, we have
\begin{equation}\label{2.19}
\frac{1}{\varepsilon}\frac{ab}{(a+b+1)(a+b)}=\frac{1}{2(1+2\varepsilon)}
=\frac{1}{2}\left(1-2\varepsilon+O(\varepsilon^2)\right).
\end{equation}
Finally, we collect on the right the terms left out in the two sums, namely
\[
\frac{1}{\varepsilon}\left(1-2\frac{(a)_n}{(a+b)_n}\right)B_n(x).
\]
We expand
\begin{align*}
2\frac{(a)_n}{(a+b)_n}&=\frac{2(\varepsilon)_n}{(2\varepsilon)_n}
=\prod_{j=1}^{n-1}\frac{\varepsilon+j}{2\varepsilon+j} \\
&=\prod_{j=1}^{n-1}\left(1+\frac{\varepsilon}{j}\right)\left(1+\frac{2\varepsilon}{j}\right)^{-1}
=\prod_{j=1}^{n-1}\left(1-\frac{\varepsilon}{j}+\frac{2\varepsilon^2}{j^2}+\dots\right) \\
&=1-H_{n-1}\varepsilon+\left(2H_{n-1}^{(2)}+\sum_{1\leq j<k\leq n-1}\frac{1}{jk}\right)\varepsilon^2+O(\varepsilon^3).
\end{align*}
Now the double sum in the last term can clearly be written as 
$(H_{n-1}^2-H_{n-1}^{(2)})/2$, and thus
\begin{equation}\label{2.20}
\frac{1}{\varepsilon}\left(1-2\frac{(a)_n}{(a+b)_n}\right)
=H_{n-1}-\left(\frac{3}{2}H_{n-1}^{(2)}+\frac{1}{2}H_{n-1}^2\right)\varepsilon+O(\varepsilon^2).
\end{equation}
If we substitute \eqref{2.17}--\eqref{2.20} into \eqref{2.5} and let 
$\varepsilon\rightarrow 0$, we recover the polynomial analogue of Miki's 
identity. Finally, if we equate the coefficients of $\varepsilon$, we 
immediately get \eqref{2.16} after multiplying both sides by $-1$ and 
exploiting symmetry in \eqref{2.17}.
\end{proof}

\subsection{Consequences of Theorem~2}
We restrict our attention to the case $k=3$ and $a_1=a_2=a_3=\varepsilon$.
Furthermore, to avoid double indices, we set $i=l_1$, $j=l_2$, $l=l_3$ on the
left, and $i=l_0$, $j=l_1$, $l=l_2$ on the right of \eqref{4.0}. Then we get,
after dividing by $n!$, 
\begin{gather}
\sum_{i+j+l=n}\frac{(\varepsilon)_i(\varepsilon)_j(\varepsilon)_l}
{(3\varepsilon)_n}\frac{B_i(x)}{i!}\frac{B_j(x)}{j!}\frac{B_l(x)}{l!}
=3\sum_{i+j+l=n}\frac{\varepsilon(\varepsilon)_j(\varepsilon)_l}
{(3\varepsilon)_{j+l+1}}\frac{B_i(x)}{i!}\frac{B_j}{j!}\frac{B_l}{l!}\label{6.9}\\
+3\sum_{i+j=n-1}\frac{\varepsilon^2(\varepsilon)_j}{(3\varepsilon)_{j+2}}
\frac{B_i(x)}{i!}\frac{B_j}{j!}
+\frac{\varepsilon^3}{(3\varepsilon)_3}\frac{B_{n-2}(x)}{(n-2)!},\nonumber
\end{gather} 
valid for all $n\geq 2$ and $\varepsilon>0$. This will be the basis for the
various results in this subsection, and also immediately gives \eqref{4.0a}.

For a first easy consequence we let $\varepsilon\rightarrow\infty$ on both
sides of \eqref{6.9}. Then the limit of the four fractions involving 
$\varepsilon$ are easily seen to be $3^{-n}$, $3^{-j-l-1}$, $3^{-j-2}$, and
$3^{-3}$, respectively. Thus, after multiplying both sides by $3^nn!$, we
immediately get the following result.

\begin{corollary}
For integers $n\geq 2$ we have
\begin{gather}
\sum_{i+j+l=n}\binom{n}{i,j,l}B_i(x)B_j(x)B_l(x)
=\sum_{i+j+l=n}\binom{n}{i,j,l}3^iB_i(x)B_jB_l\label{6.10} \\
+n\sum_{i=0}^{n-1}\binom{n-1}{i}3^iB_i(x)B_{n-1-i}+n(n-1)3^{n-3}B_{n-2}(x).\nonumber
\end{gather}
\end{corollary}

For the next consequence of \eqref{6.9} we set $x=0$, for greater simplicity
of the statement. The proof is tedious, and we leave the details to the 
interested reader. 

\begin{corollary}
For integers $n\geq 2$ we have
\begin{align*}
&\frac{1}{3}\sum_{\substack{i+j+l=n\\i,j,l\geq 1}}
\frac{B_i}{i}\frac{B_j}{j}\frac{B_l}{l}
=\sum_{\substack{i+j+l=n\\i,j,l\geq 1}}\binom{n-1}{i-1}
\frac{B_i}{i}\frac{B_j}{j}\frac{B_l}{l} 
+\sum_{l=1}^{n-2}\binom{n-1}{l+1}\frac{B_l}{l}\frac{B_{n-l-1}}{n-l-1}\\
&+\sum_{l=1}^{n-1}\left(3H_{n-1}-2H_{l-1}+\tfrac{1}{n}\right)\frac{B_l}{l}\frac{B_{n-l}}{n-l}
-2\sum_{l=1}^{n-1}\binom{n-1}{l-1}\left(2H_l+\tfrac{1}{l}\right)\frac{B_l}{l}\frac{B_{n-l}}{n-l} \\
&+\frac{n-1}{6}B_{n-2}+\left(\frac{1}{(n-1)n}-3\right)B_{n-1} 
-2\left(\frac{2}{n}H_{n-1}+H_{n-1}^2+2H_{n-1}^{(2)}+\frac{3}{n^2}\right)\frac{B_n}{n}.
\end{align*}
\end{corollary}
This can be seen as a third-order analogue of Miki's identity.
Note the difference in complexity between this result and the third-order
analogue of Matiyasevich's identity given in \eqref{4.0a}. See also 
\cite[(6.5)]{AD} for a third-order ``Miki analogue" for Bernoulli polynomials.
To prove Corollary 9, one can use a similar method as in the proof of 
Corollary~7, and proceed as follows:

-- Collect the ``edge" and ``corner" terms in the sums of \eqref{6.9}.

-- Divide both sides of \eqref{6.9} by $\varepsilon^2$.

-- Expand the various fractions involving $\varepsilon$ in a similar
way as in \eqref{2.17} and \eqref{2.18}.

-- Equate the constant terms (i.e., let $\varepsilon\rightarrow 0$) to
obtain a first identity.

-- Equate the coefficients of $\varepsilon$ to obtain a second identity.

\noindent
Interestingly, in this case the first identity turns out to be equivalent to
\eqref{1.3}, Miki's original identity. The second one is Corollary 9.

\subsection{Consequences of Theorems~3 and~4}

Given the similarities between Theorems~1 and~3, it is clear that Euler 
analogues of Corollaries~2--7 could easily be derived; recall that an analogue
of Corollary~1 is already stated following Theorem~3. Here we restrict 
ourselves to only a few more consequences; we also skip the proofs which are
again similar to the proof of Corollary~7.

\begin{corollary}
For integers $n\geq 2$ we have
\begin{align*}
\sum_{l=1}^{n-2}\frac{E_l(x)}{l}\frac{E_{n-l-1}(x)}{n-l-1}
&=4\sum_{l=1}^{n-1}\binom{n-2}{l-1}H_{l-1}\frac{B_{n-l}(x)}{n-l}\frac{E_l(0)}{l}\\
&\qquad\qquad+2H_{n-2}\frac{E_{n-1}(x)}{n-1}+4\frac{H_{n-1}}{n-1}\frac{E_n(0)}{n},
\end{align*}
and for $n\geq 1$,
\begin{align*}
&\sum_{l=1}^{n-1}\left(H_{n-1}-H_{l-1}\right)
\frac{E_l(x)}{l}\frac{E_{n-l}(x)}{n-l} 
=\frac{1}{2}\left(H_{n-1}^2+3H_{n-1}^{(2)}\right)\frac{E_n(x)}{n} \\
&\quad+\sum_{l=1}^n\binom{n-1}{l-1}\left(H_{l-1}^2+3H_{l-1}^{(2)}\right)
\frac{B_{n+1-l}(x)}{n+1-l}\frac{E_l(0)}{l}
+\frac{1}{n}\left(H_n^2+3H_n^{(2)}\right)\frac{E_{n+1}(0)}{n+1}.
\end{align*}
\end{corollary}

Finally, to obtain third-order analogues of Miki's identity, we start with
\eqref{5.26} for $k=3$ and follow the outline described after Corollary~9.
This leads to the following identities.

\begin{corollary}
For integers $n\geq 2$ we have
\begin{align*}
\sum_{l=1}^{n-1}&\left(\frac{E_l(x)}{l}\frac{E_{n-l}(x)}{n-l}
-\frac{E_l(0)}{l}\frac{E_{n-l}(0)}{n-l}\right) \\
&\qquad\qquad=\sum_{\substack{i+j+l=n\\i,j,l\geq 1}}\binom{n-1}{i}
E_i(x)\frac{E_j(0)}{j}\frac{E_l(0)}{l}+2H_{n-1}\frac{E_n(x)}{n},
\end{align*}
and 
\begin{align*}
\frac{1}{3}&\sum_{\substack{i+j+l=n\\i,j,l\geq 1}}
\frac{E_i(x)}{i}\frac{E_j(x)}{j}\frac{E_l(x)}{l} 
=-2\left(H_{n-1}^2+2H_{n-1}^{(2)}\right)\frac{E_n(x)}{n} \\
&+\sum_{\substack{i+j+l=n\\i,j,l\geq 1}}\binom{n-1}{i}
\left(H_{j-1}+H_{l-1}-3H_{j+l-1}\right)E_i(x)\frac{E_j(0)}{j}\frac{E_l(0)}{l}\\
&+\sum_{l=1}^{n-1}\left(3H_{n-1}-H_{l-1}-H_{n-l-1}\right)
\left(\frac{E_l(x)}{l}\frac{E_{n-l}(x)}{n-l}
-\frac{E_l(0)}{l}\frac{E_{n-l}(0)}{n-l}\right).
\end{align*}
\end{corollary}

Using the special values \eqref{5.6}, numerous identities involving Genocchi
and/or Euler numbers could also be obtained.

\section{Final remarks}

{\bf 1.} By choosing different values of the parameters in Theorems~1--4,
many more identities for Bernoulli and Euler numbers and polynomials could be
obtained, some relatively simple, and others of increasing complexity.
We have shown in this paper that the original identities of Miki and 
Matiyasevich and their various extensions are special cases of very general
class of identities.

\medskip
{\bf 2.} This is not the first common extension of the identities of Miki and
Matiyasevich. In fact, using generating functions, Dunne and Schubert \cite{DS}
recently proved the following result.

\begin{theorem}[Dunne and Schubert]
For any integer $n\geq 2$ and real $p\geq 0$ we have
\begin{align}
\sum_{l=1}^{n-1}(2l)_p(2n-2l)_p&\frac{B_{2l}}{2l}\frac{B_{2n-2l}}{2n-2l}
= 2B_{2n}\frac{\Gamma(2n+2p)}{(2n)!}\sum_{l=1}^{2n-1}
\frac{\Gamma(p+l)\Gamma(p+1)}{\Gamma(2p+l+1)}\label{7.1} \\
&+\frac{2\Gamma(p+1)}{(2n)!}\sum_{l=1}^n\binom{2n}{2l}
\frac{\Gamma(p+2l)\Gamma(2p+2n)}{\Gamma(2p+2l+1)}B_{2l}B_{2n-2l}.\nonumber
\end{align}
\end{theorem}

The case $p=0$ gives Miki's identity, while for $p=1$ we get
\[
\sum_{l=1}^nB_{2l}B_{2n-2l} = \frac{1}{n+1}
\sum_{l=1}^n\binom{2n+2}{2l+2}B_{2l}B_{2n-2l} + 2nB_{2n},
\]
which is equivalent to Matiyasevich's identity. 

The identity \eqref{7.1} actually follows from Theorem~1 if we take $a=b=p$ and
$x=0$, then replace $n$ by $2n$ and extract the end terms in the sums. After
some manipulations we then get
\begin{align}
\sum_{l=1}^{n-1}(2l)_p(2n-2l)_p&\frac{B_{2l}}{2l}\frac{B_{2n-2l}}{2n-2l}
= \frac{\Gamma^2(p)}{(2n)!}\left((2p)_{2n}-2(p)_{2n}\right) \label{7.2} \\
&+\frac{2\Gamma(p+1)}{(2n)!}\sum_{l=1}^n\binom{2n}{2l}
\frac{\Gamma(p+2l)\Gamma(2p+2n)}{\Gamma(2p+2l+1)}B_{2l}B_{2n-2l}.\nonumber
\end{align}
Comparing \eqref{7.1} with \eqref{7.2} shows that after some simplification
we have 
\[
\sum_{l=1}^{2n-1}\frac{\Gamma(p+l)}{\Gamma(2p+l+1)}
=\frac{\Gamma(p)}{\Gamma(2p+1)} - \frac{\Gamma(p+2n)}{p\Gamma(2p+2n)}.
\]
This last identity can be proved independently, for instance by manipulating
the integral representation of Euler's beta function.


\begin{thebibliography}{25}

\bibitem{AS} M. Abramowitz and I. A. Stegun, {\it Handbook
of Mathematical Functions}, National Bureau of Standards, 1964. 

\bibitem{Ag} T.~Agoh, {\it Convolution identities for Bernoulli and 
Genocchi polynomials\/}. Electron. J. Combin. {\bf 21} (2014), \#P1.65, 14pp.

\bibitem{AD} T.~Agoh and K.~Dilcher, {\it Higher-order convolutions for 
Bernoulli and Euler polynomials\/}. J. Math. Anal. Appl. {\bf 419} (2014), 
no.~2, 1235--1247.

\bibitem{Ch} W.~Chu, {\it Reciprocal formulae for convolutions of Bernoulli 
and Euler polynomials\/}. Rend. Mat. Appl. (7) {\bf 32} (2012), no.~1--2, 
17--73.

\bibitem{DSS} K. Dilcher, L.~Skula, and I.~Sh.~Slavutskii,
{\it Bernoulli Numbers. Bibliography (1713-1990)\/},
Queen's Papers in Pure and Applied Mathematics, 87, Queen's University,
Kingston, Ont., 1991. Updated on-line version:
{\tt http://www.mathstat.dal.ca/\~{}dilcher/bernoulli.html}.

\bibitem{DS} G.~V.~Dunne and C.~Schubert, {\it Bernoulli number
identities from quantum field theory and topological string theory\/}.
Commun. Number Theory Phys. {\bf 7} (2013), no.~2, 225--249.

\bibitem{Du} R.~Durrett, {\it Probability: Theory and Examples\/}, Wadsworth
\& Brooks/Cole, Pacific Grove, Ca, 1991.

\bibitem{EMOT} A.~Erd\'elyi, W.~Magnus, F.~Oberhettinger, and F.~G.~Tricomi, 
{\it Higher transcendental functions\/}, Vol. I. Based, in part, on notes 
left by Harry Bateman. McGraw-Hill Book Company, Inc., New York-Toronto-London,
1953.

\bibitem{Ge1} I.~M.~Gessel, {\it Applications of the classical umbral
calculus\/}. Algebra Universalis {\bf 49} (2003), 397--434.

\bibitem{Ge2} I.~M.~Gessel, {\it On Miki's identity for Bernoulli numbers\/}.
J. Number Theory {\bf 110} (2005), 75--82.

\bibitem{GKP} R.~L.~Graham, D.~E.~Knuth, and O.~Patashnik,
{\it Concrete Mathematics\/}, Addison-Wesley Publ. Co., Reading, MA, 1989.

\bibitem{Ha} E.~R.~Hansen, {\it  A Table of Series and Products\/},
Prentice-Hall, Inc., Englewood Cliffs, NJ, 1975.

\bibitem{JK} N.~L.~Johnson and S.~Kotz, {\it Distributions in Statistics:
Continuous Multivariate Distributions\/}, John Wiley \& Sons, New York, 1972.

\bibitem{Jo} C.~Jordan, {\it Calculus of finite differences\/}, 2nd ed., 
Chelsea Publ. Co., New York, 1950.

\bibitem{Lu} E.~Lukacs, {\it A characterization of the gamma distribution\/}. 
Ann. Math. Statist. {\bf 26} (1955), 319--324. 

\bibitem{Ma} Yu.~Matiyasevich, {\it Identities with Bernoulli numbers\/}.\\
{\tt http://logic.pdmi.ras.ru/\~{}yumat/personaljournal/identitybernoulli/bernulli.htm}

\bibitem{Mi} H.~Miki, {\it A relation between Bernoulli numbers\/}. 
J. Number Theory {\bf 10} (1978), no.~3, 297--302. 

\bibitem{No} N.~E.~N\"orlund, {\it Vorlesungen \"uber Differenzenrechnung\/},
Springer-Verlag, Berlin, 1924.

\bibitem{DLMF} F.~W.~J.~Olver et al. (eds.), {\it NIST Handbook of Mathematical
Functions\/}, Cambridge Univ. Press, New York, 2010.
Online companion: {\tt http://dlmf.nist.gov/}.

\bibitem{PS} H.~Pan and Z.-W.~Sun, {\it New identities involving Bernoulli 
and Euler polynomials\/}. J. Combin. Theory Ser. A {\bf 113} (2006), no. 1, 
156--175.

\bibitem{RT} G.-C.~Rota and B.~D.~Taylor, {\it The classical umbral calculus\/}.
SIAM J. Math. Anal. {\bf 25} (1994), no.~2, 694--711. 

\bibitem{SW} P.~Sun and T.~M.~Wang, {\it A probabilistic interpretation to 
umbral calculus\/}. J. Math. Res. Exposition {\bf 24} (2004), no.~3, 391--399.

\bibitem{Wa} J.~B.~Walsh, {\it Knowing the Odds. An Introduction to 
Probability\/}, American Mathematical Society, Providence, RI, 2012.


\end{thebibliography}
\end{document}